\documentclass[12pt]{amsart}
\usepackage{geometry}

\usepackage[mathscr]{euscript}
\usepackage[all]{xy}
\usepackage[T1]{fontenc}
\usepackage{amsfonts}

\usepackage{subcaption}
\usepackage{amssymb}
\usepackage{mathrsfs}
\usepackage[dvipsnames,usenames]{xcolor}
\usepackage[colorlinks=true,urlcolor=blue,linkcolor=blue,citecolor=blue]{hyperref}
\usepackage{upgreek}
\hypersetup{
	linkbordercolor={1 0 0}, 
	citebordercolor={0 1 0} 
}
\usepackage{cite}
\usepackage[noabbrev]{cleveref}

\usepackage{tikz}

\newcommand{\al}{\alpha}
\newcommand{\be}{\beta}
\newcommand{\ga}{\gamma}

\newcommand{\ep}{\varepsilon}

\newcommand{\la}{\lambda}
\newcommand{\si}{\sigma}

\newcommand{\Si}{\Sigma}
\newcommand{\ZZ}{{\mathbb Z}}

\newcommand{\QQ}{{\mathbb Q}}
\newcommand{\RR}{{\mathbb R}}

\newcommand{\RP}{{\mathbb R}{\rm P}}

\newcommand{\fN}{{\mathfrak N}}

\newcommand{\cG}{\mathcal G}
\newcommand{\cK}{\mathcal K}

\newcommand{\tr}{\mathsf{T}}

\newcommand{\lto}{\longrightarrow}
\newcommand{\lk}{\operatorname{\ell{\it k}}}
\newcommand{\sig}{\operatorname{sig}}
\newcommand{\sm}{\smallsetminus}
\newcommand{\co}{\colon}

\newcommand{\wt}{\widetilde}
 \newcommand{\Arf}{\operatorname{Arf}}
\newcommand{\Br}{{\upbeta}}
\newcommand{\Ker}{\operatorname{Ker}}

\newcommand*\wbar[1]{
  \hbox{ \kern-0.2em%
    \vbox{%
      \hrule height 0.5pt  
      \kern0.25ex
      \hbox{%
        \kern-0.10em
        \ensuremath{#1}%
        \kern-0.05em
      }%
    }%
  \kern0.05em}%
}
\makeatletter
\newcommand*\bigcdot{\mathpalette\bigcdot@{.55}}
\newcommand*\bigcdot@[2]{\mathbin{\vcenter{\hbox{\scalebox{#2}{$\m@th#1\bullet$}}}}}
\makeatother

\newtheorem{theorem}{Theorem} [section]
\newtheorem{lemma}[theorem]{Lemma}
\newtheorem{proposition}[theorem]{Proposition}

\theoremstyle{definition}     
\newtheorem{definition}[theorem]{Definition}

\theoremstyle{remark}
\newtheorem{remark}[theorem]{Remark}
\newtheorem{example}[theorem]{Example}

\title[Concordance invariants of null-homologous knots]{Concordance invariants of null-homologous \\ knots in thickened surfaces}
\author[H. U. Boden]{Hans U. Boden}
\address{Mathematics \& Statistics, McMaster University, Hamilton, Ontario}
\email{boden@mcmaster.ca}

\author[H. Karimi]{Homayun Karimi}
\address{Mathematics \& Statistics, McMaster University, Hamilton, Ontario}
\email{karimih@math.mcmaster.ca}


\begin{document}

\begin{abstract}
Using the Gordon-Litherland pairing, one can define invariants (signature, nullity, determinant) for $\ZZ/2$ null-homologous links in thickened surfaces. In this paper, we study the concordance properties of these invariants.  For example, if $K \subset \Si \times I$ is $\ZZ/2$ null-homologous and slice, we show that its signatures vanish and its determinants are perfect squares. 
These statements are derived from a cobordism result for closed unoriented surfaces in certain 4-manifolds. 

The Brown invariants are defined for $\ZZ/2$ null-homologous links in thickened surfaces. They take values in $\ZZ/8 \cup \{\infty\}$ and depend on a choice of spanning surface. 
We present two equivalent methods to defining and computing them, and we prove a chromatic duality result relating the two.  We study their concordance properties, and we show how to interpret them as Arf invariants for null-homologous links. 
The Brown invariants and knot signatures are shown to be invariant under concordance  of spanning surfaces.
\end{abstract}

\maketitle
\section*{Introduction}
In the previous paper \cite{Boden-Chrisman-Karimi-2021}, we introduced  invariants (signature, determinant, nullity) for  $\ZZ/2$ null-homologous links in thickened surfaces. The invariants are defined in terms of the Gordon-Litherland pairing, and they depend on a choice of spanning surface up to $S^*$-equivalence. 

In the present paper, we use the invariants to define slice obstructions for knots in thickened surfaces. Specifically, given a $\ZZ/2$ null-homologous knot $K$ in a thickened surface, we show that if $K$ is slice, then its knot signatures vanish and its knot determinants are perfect squares. 
(See Theorems \ref{slice} and \ref{determinant}.) We also show that the knot signatures are invariant under a new notion of concordance, namely concordance of spanning surfaces (see \Cref{defn-spanning-surf}).

These results are deduced from our main result, \Cref{boundary}.
Suppose $W$ is a compact, oriented 3-manifold with $\partial W = \Si$, and $E \subset W \times I$ is a closed unoriented surface with $[E]=0$ in $H_2(W\times I;\ZZ/2)$ and trivial Euler class. Then  \Cref{boundary} asserts that there exists a 3-manifold $V$ embedded in $W \times I$ with $\partial V=E.$  It is used to prove Theorems \ref{slice} and \ref{determinant}, and it is also used to prove a vanishing result for the Brown invariants for slice knots (\Cref{cobordism}).

The Brown invariants are defined for $\ZZ/2$ null-homologous links in thickened surfaces and take values in $\ZZ/8 \cup \{\infty \}$. We provide two definitions, one in terms of $\ZZ/4$ quadratic enhancements of the Gordon-Litherland pairing, and another in terms of Goeritz matrices associated to checkerboard colorings of link diagrams. We relate the two approaches by proving a chromatic duality result (\Cref{thm:chromatic}). 

The first approach shows that the Brown invariants depend on a choice of spanning surface up to $S^*$-equivalence. (A proof that they are invariant under $S^*$-equivalence can be found in the recent paper of Klug \cite{Klug-2020}.) Every non-split, $\ZZ/2$ null-homologous link $L \subset \Si \times I$ has two distinct $S^*$-equivalence classes of spanning surfaces, thus each such link admits two distinct Brown invariants. In this way, the Brown invariants are similar to the other invariants (signature, determinant, nullity) defined in terms of the Gordon-Litherland pairing. 

The second approach shows that the Brown invariants can always be computed in terms of Goeritz matrices. This leads to a fast and efficient algorithm for computing them. We also study the behavior of the invariants under horizontal and vertical mirror symmetry, and we relate them to Arf invariants in the special case when $L\subset \Si \times I$ is a null-homologous link.

The Goeritz matrices were first introduced in \cite{Goeritz}, which gave the first applications of quadratic forms to knot theory, see \cite{Przytycki-2011} and \cite{Traldi-2017}. Goeritz matrices continue to inspire new and important results, such as \cite{Boninger} and \cite{Traldi}. In \cite{Boninger}, Boninger shows that the Jones polynomial of a link can be computed from its Goeritz matrix. In \cite{Traldi}, Traldi shows that a link is determined up to mutation by the Goeritz matrices of its diagrams.

Our results touch upon another invariant with a long and distinguished history, namely the Arf invariant \cite{Arf-I}. As a knot invariant, it was first studied by Robertello \cite{Robertello-1965}, who showed it to be invariant under concordance. It is also invariant under the band pass move \cite{Kauffman-1987}, and it is equal to the mod 2 reduction of Casson's knot invariant \cite{Polyak-Viro}. It is also equal to the value $V_K(t)\vert_{t=i}$ of the Jones polynomial  \cite{Lickorish}. In particular, the Arf invariant is an invariant of finite-type, in fact, it is the only  finite-type knot invariant which is also invariant under concordance \cite{Ng}.


One motivation for studying concordance of knots in thickened surfaces comes from applications to concordance of virtual knots. For example, the Brown invariants extend to invariants of checkerboard colorable virtual knots and links. In fact, using parity projection, they can be extended to invariants for all virtual knots.
As well, since parity projection preserves concordance (Theorem 5.9 \cite{Boden-Chrisman-Gaudreau-2020}), it follows that the extended Brown invariants are slice obstructions for arbitrary virtual knots.

For a null-homologous knot, the Brown invariant specializes to its Arf invariant, so it is natural to wonder whether the Brown invariants are finite-type invariants. Specifically, can they be computed on subdiagrams? For classical knots, the Casson invariant is an integral lift of the Arf invariant, and so it is natural to wonder whether the Brown invariants also have integral lifts.

For a classical link, it is well-known that the Arf invariant is only defined when the link is \textit{proper}. (Recall that a classical link  $L= K_1 \cup \cdots \cup K_n$ is proper if $\lk(K_i,L\sm K_i)$ is even for all $i=1,\ldots, n.$) This is precisely the condition needed to ensure that the quadratic form associated to a Seifert surface for $L$ is proper, see \cite[p.226]{Kirby-Melvin}. It would be interesting to determine conditions on a link with spanning surface $F \subset \Si \times I$ so that its $\ZZ/4$ quadratic form $\varphi_F$ is \textit{proper} (see \Cref{section-Brown}).


In \Cref{sec1}, we review concordance for links $L \subset \Si \times I$ and discuss the Gordon-Litherland pairing and the resulting link invariants.
In  \Cref{sec2}, we state and prove the main result, \Cref{boundary}. We also provide a discussion focused on the hypotheses of \Cref{boundary}. 
In \Cref{sec3}, we establish the slice criteria on the signature and determinant derived from \Cref{boundary}.
In  \Cref{sec4}, we review $\ZZ/4$ enhanced forms and their Brown invariants. We then discuss the associated link invariants, showing they provide slice obstructions and proving a chromatic duality result. In  \Cref{sec5}, we introduce the notion of  concordance of spanning surfaces, and we study the behavior of the signature and Brown invariants under concordance of spanning surfaces.

\medskip\noindent
{\textbf{Conventions.}}
Throughout this paper, spanning surfaces are assumed to be compact and connected, but they are not assumed to be oriented or even orientable.
\section{Preliminaries} \label{sec1}
We begin this section by introducing some basic notions. We then review Turaev's definition of concordance for links in thickened surfaces, and recall the construction of the Gordon-Litherland pairing and describe link invariants such as the signature, nullity, determinant which are defined in terms of the Gordon-Litherland pairing.

\subsection{Basic notions}
Throughout this paper, $\Si$ will denote a compact, connected, oriented surface and $I = [0,1]$, the unit interval.  A link in $\Si \times I$ is an embedding of $S^1 \sqcup \cdots \sqcup S^1$ into the interior of $\Si \times I$, considered up to orientation-preserving homeomorphisms of the pair $(\Si \times I, \Si \times \{0\})$. The link $L \subset \Si \times I$ is said to be $R$ \textit{null-homologous} if $[L]=0$ in $H_1(\Si\times I; R)$. When $R=\ZZ$ is understood by context, we use null-homologous to mean  $\ZZ$ null-homologous without any confusion.

Given a link $L \subset \Si \times I$, a \textit{spanning surface} for $L$ is a compact, connected, unoriented surface $F$ with boundary $\partial F =L$. If $F$ is oriented, then it is called a \textit{Seifert surface} for $L$. In that case, $L$ inherits an orientation as the oriented boundary of $F$.
 
Note that a link $L \subset \Si \times I$ is $\ZZ/2$ null-homologous if and only if it admits a spanning surface, and it is null-homologous if and only if it admits a Seifert surface.

Any spanning surface $F\subset \Si \times I$ for $L$ determines  a symmetric bilinear pairing on $H_1(F;\ZZ)$ called the \textit{Gordon-Litherland} pairing.  Using this pairing, one can define signature, determinant and nullity invariants  for $\ZZ/2$ null-homologous links $L \subset \Si \times I$, \cite{Boden-Chrisman-Karimi-2021}. When $L$ is checkerboard colorable, the invariants of \cite{Boden-Chrisman-Karimi-2021} agree with the invariants defined by Im, Lee, and Lee in terms of Goeritz matrices \cite{Im-Lee-Lee-2010}.

The invariants (signature, determinant, nullity) derived from the Gordon-Lither\-land pairing depend on the choice of spanning surface, but they are invariant under $S^*$-equivalence (see \Cref{equivalence} and \cite[\S 2.3]{Boden-Chrisman-Karimi-2021}). 

When $\Si$ has genus $g(\Si) \geq 1$, every non-split $\ZZ/2$ null-homologous link $L \subset \Si \times I$ has exactly two $S^*$-equivalence classes of spanning surfaces, see \cite[Proposition 1.6]{Boden-Chrisman-Karimi-2021}. Thus, links in thickened surfaces typically have two signatures, two determinants, and two nullities.

\subsection{Concordance and slice knots}
We recall the notions of cobordism and concordance for knots in thickened surfaces,
originally introduced by Turaev \cite{Turaev-2008-a}.

\begin{definition}[Turaev] \label{defn:conc}
Two knots $K_0 \subset \Si_0 \times I$ and $K_1 \subset \Si_1\times I$ are said to be \textit{concordant} if there exists a compact, oriented 3-manifold $W$ with $\partial W =-\Si_0 \cup \Si_1$ and an annulus $A$  properly embedded in $W \times I$ such that $\partial A = K_0 \sqcup K_1$. 

More generally, a \textit{cobordism} is an oriented surface $Z$ properly embedded in $W \times I$ such that $\partial Z = K_0 \sqcup K_1$. 
\end{definition}

Note that there exists a cobordism between any two knots in thickened surfaces. For a proof, see \cite{Kauffman-2015}. The \textit{slice genus} of a knot $K \subset \Si \times I$ is defined by setting $$g_s(K) = \min \{g(Z) \mid Z \text{ is a cobordism from $K$ to the unknot} \},$$ where $g(Z)$ denotes the genus of cobordism surface $Z$. A knot $K \subset \Si \times I$ is said to be \textit{slice} if it is concordant to the unknot. In particular, a knot $K \subset \Si \times I$ is slice if and only if $g_s(K)=0$.

\subsection{Gordon-Litherland pairing}
Associated to a link $L$ in $\Si \times I$ with spanning surface $F$, there is a symmetric bilinear pairing on $H_1(F;\ZZ)$ called the \textit{Gordon-Litherland pairing}. In this section, we review the definition of the Gordon-Litherland pairing, following \cite{Boden-Chrisman-Karimi-2021} and \cite{Boden-Karimi-2020}.

To begin, we recall the asymmetric linking for simple closed curves in $\Si \times I$. Given two disjoint oriented simple closed curves $J, K$  in the interior of $\Si \times I$, define $\lk(J, K) = J \cdot B,$ where $B$ is a 2-chain in $\Si \times I$ such that $\partial B =  K - v$ for some  1-cycle $v$ in $\Si \times \{1\}$ and $\cdot$ denotes the intersection number.

Let $F \subset \Si \times I$ be a compact, connected, unoriented surface. Its normal bundle $N(F)$ has boundary a $\{\pm 1\}$-bundle $\widetilde{F}\stackrel{\pi}{\lto}F$, a double cover with $\widetilde{F}$ oriented. Define the transfer map $\tau \co H_1(F;\ZZ) \to H_1(\widetilde{F};\ZZ)$ by setting $\tau([a]) = [\pi^{-1}(a)].$

For $a,b\in H_1(F;\ZZ)$, define  $\cG_F(a,b)=\tfrac{1}{2}\big(\lk(\tau a,b)+\lk(\tau b,a)\big).$ This pairing is well-defined, takes values in the integers, and is symmetric (for proofs of these and other statements, see \cite{Boden-Chrisman-Karimi-2021} and \cite{Boden-Karimi-2020}). The map $\cG_F \co H_1(F;\ZZ) \times H_1(F;\ZZ) \to \ZZ$ is called the \textit{Gordon-Litherland pairing}.

Let $L \subset \Si \times I$ be a $\ZZ/2$ null-homologous link and $F \subset \Si \times I$ a spanning surface for $L$.
Assuming that $L$ has $m$ components, we can write $L=K_1 \cup \dots \cup K_m$. Let $L' =K_1' \cup \dots \cup K_m'$ be the push-off of $L$ in $\Si \times I$ that misses $F$. Fix an orientation on $L$, which gives an orientation for each component $K_i$, and choose the compatible orientation on $K_i'$. Then define 
\begin{equation*}
e(F) = -\sum_{i=1}^m \lk(K_i, K_i') \quad \text{and} \quad
e(F,L) =  -\sum_{i,j=1}^m \lk(K_i, K_j').
\end{equation*}
Here, $e(F,L)$ will depend on the choice of orientation of $L$, but $e(F)$ is independent of that choice. The two are related by the formula $e(F,L) = e(F) - \la(L)$, where $\la(L) = \sum_{i \neq j} \lk(K_i, K_j)$ is the total linking number of $L$. In the case of knots, $e(F,K)=e(F).$

In \cite{Boden-Chrisman-Karimi-2021}, the Gordon-Litherland pairing is used to define signature, determinant, and nullity invariants for  links in thickened surfaces. 
Let $L \subset \Si \times I$ be a link and $F \subset \Si \times I$ a spanning surface. Then the determinant is  $\det(L,F)=|\det(\cG_{F})|$, the nullity is $n(L,F)=\text{nullity}(\cG_{F})$, and the signature is $\si(L,F)=\sig(\cG_{F})+\frac{1}{2} e(F,L)$. Each of $\det(L,F), n(L,F)$ and $\si(L,F)$  depend only on the $S^*$-equivalence of the spanning surface $F$. 

\begin{definition} \label{equivalence}
Two spanning surfaces are \textit{$S^*$-equivalent} if one can be obtained from the other by (i) ambient isotopy, (ii) attachment (or removal) of a tube, and (iii) attachment (or removal) of a half-twisted band.
\end{definition}

A link in $\Si \times I$ is said to be \textit{split} if it can be represented by a disconnected diagram on $\Si$. A link is said to be \textit{checkerboard colorable} if it can be represented by a checkerboard colorable diagram on $\Si$ (see \Cref{defn:123}).

As previously mentioned, for every link $L \subset \Si \times I$ that is non-split and $\ZZ/2$ null-homologous, there are two $S^*$-equivalence classes of spanning surfaces. In fact, a link $L \subset \Sigma \times I$  is  $\ZZ/2$ null-homologous if and only if it is checkerboard colorable (see \cite[Proposition 1.1]{Boden-Chrisman-Karimi-2021}), and every spanning surface is $S^*$-equivalent to one of checkerboard surfaces (see \cite[Proposition 1.6]{Boden-Chrisman-Karimi-2021}). Thus, for such links, there are two sets of invariants arising from the Gordon-Litherland pairing.

\section{Main theorem} \label{sec2}
In this section, we state and prove the main theorem, which gives a cobordism result for closed unoriented surfaces. The proof involves obstruction theory, and we are grateful to Danny Ruberman, who sent us a sketch of a key step in the following proof. Following the proof, we discuss applications with an eye toward ensuring that the hypotheses of \Cref{boundary} are satisfied.

Applications of \Cref{boundary} to questions about concordance of knots will be given in \Cref{sec3}. This is where we show that the signature and determinant provide obstructions to knots being slice. 

For additional background information about nonorientable surfaces in 3- and 4-manifolds, we refer readers to \cite{Bredon-Wood} and \cite{Levine-Ruberman-Strle}.



\begin{theorem}\label{boundary}
Let $W$ be a compact oriented 3-manifold with $\partial W=\Si$. Assume that $E\subset W\times I$ is a closed nonorientable surface with $[E]=0$ in $H_2(W\times I; \ZZ/2)$ and with normal Euler number $e(E)=0$. Then there exists a compact 3-manifold $V\subset W\times I$ with $\partial V=E$.
\end{theorem}

\begin{proof}
Every nonorientable surface $E$ can be written as a connected sum of $\RP^2$, and in this context, a surface with $E = \#_{i=1}^{g} \RP^2$ is said to have nonorientable genus $g$. A routine calculation shows that 
$$H_i(E;\ZZ)= \begin{cases} \ZZ & \text{if $i=0,$}\\
\ZZ/2 \oplus \ZZ^{g-1}& \text{if $i=1,$}
\end{cases}\quad \text{and} \quad
H_i(E;\ZZ/2)= \begin{cases} \ZZ/2 & \text{if $i=0,2,$}\\
(\ZZ/2)^{g}& \text{if $i=1.$}
\end{cases}$$

In this case, the assumption that $e(E)=0$ implies that $E$ has nonorientable genus $g=2k$. Thus $E$ abstractly bounds a 3-manifold. This fact, together with the assumption that $[E]=0$ in $H_2(W \times I;\ZZ/2),$ implies that $E$ is trivial in $\fN_2(W \times I)$, the unoriented bordism group, by Theorem 17.1 of \cite{Conner-Floyd}. It remains to show that $E$ bounds a 3-manifold embedded in $W \times I$.

Let $N_E$ be a tubular neighborhood of $E$ in $W\times I$ and let $S_E$ be its boundary. Then $N_E$ is a $D^2$ bundle over $E$ isomorphic to the normal bundle $\nu(E)$  to the embedding $j \co E \to W\times I$, and $S_E$ is the associated $S^1$-bundle. Since $\nu(E)\oplus TE \cong j^*(T(W\times I))$ and $W\times I$ is orientable, it follows that $w_1(\nu(E)) = w_1(E).$

Although the normal bundle $\nu(E)$ is not orientable, it has an Euler class in $H^{2}(E;\ZZ_w),$ cohomology with coefficients twisted by the representation $w\co \pi_1(E) \to \{\pm 1\}$ associated to the first Stiefel-Whitney class $w_1(E)$ \cite{Massey}. Poincar\'{e} duality for twisted coefficients (cf.~Theorem 5.7 in \cite{Davis-Kirk}) implies that $H^2(E; \ZZ_w) \cong H_0(E;\ZZ).$ Thus the Euler class is determined by an integer, namely the normal Euler number $e(E).$ Since $e(E)=0$, the bundle $\nu(E)$ admits a nowhere zero section, which we use to split $\nu(E) = L \oplus \ep^1$. Here $\ep^1$ is trivial and $L$ is the line bundle over $E$ with $w_1(L) = w_1(E)$.

We use obstruction theory to describe the sections of $S_E$ up to homotopy. Indeed, since the higher homotopy groups of $S^1$ are all trivial, there is a single obstruction to the existence of a section of $S_E$ lying in $H^{2}(E;\ZZ_w),$ and the obstruction to finding a homotopy of sections lies in $H^{1}(E;\ZZ_w).$ Thus homotopy classes of sections of $S_E$ are in one-to-one correspondence with elements in $H^{1}(E;\ZZ_w),$ which is isomorphic to $H_1(E;\ZZ)$ by Poincar\'{e} duality for twisted coefficients.

Let $X = W\times I \smallsetminus\text{int}(N_E)$ and consider the Mayer-Vietoris sequence associated with the decomposition $W\times I = X \cup N_E$:
\begin{equation}\label{mv-seq}
\xymatrix{0 \ar[r]^{} & H_2(S_E;\ZZ/2)\ar[r]^{\varphi\qquad \quad} & H_2(N_E; \ZZ/2)\oplus H_2(X; \ZZ/2)\ar[r]^{\quad \;\; \psi} & H_2(W\times I; \ZZ/2).}
\end{equation}
Since $[E]=0$ in $H_2(W \times I;\ZZ/2),$ we have $\psi([E],0) =0$. By exactness, there must exist an element $\al \in H_2(S_E;\ZZ/2)$ with $\varphi(\al)=([E],0).$  We will show that $\al$ can be represented as the image of a section of $S_E.$

The Gysin sequence for the bundle $p \co S_E \to E$ gives that
\begin{equation}\label{gysin}
\xymatrix{0 \ar[r]^{} & H^1(E;\ZZ/2)\ar[r]^{p^{*}} & H^1(S_E; \ZZ/2)\ar[r]^{f} & H^0(E; \ZZ/2)\ar[r]^{\smallsmile\, w_2}& H^2(E;\ZZ/2).}
\end{equation}
Here the second map $f$ is given by evaluating a cohomology class on the circle fiber, and the third map  is trivial since $w_2=w_2(E)=0$. By the sequence \eqref{gysin} and our previous calculations, we have that $H^1(S_E;\ZZ/2)\cong (\ZZ/2)^{2k+1}$, thus $H_2(S_E;\ZZ/2)\cong (\ZZ/2)^{2k+1}$ by  Poincar\'{e} duality.

Given a section of $S_E$, its image determines an element of $H_2(S_E;\ZZ/2)$, which we denote $[E']$. We can write $\varphi = (\varphi_1,\varphi_2)$ for the first map of the Mayer-Vietoris sequence \eqref{mv-seq}, and note that $\varphi_1(\al) = p_*(\alpha)$ under the natural identification $H_2(N_E;\ZZ/2) \cong H_2(E;\ZZ/2)$. For any section, we have $\varphi([E'])=([E],\be)$ for some $\be \in H_2(X;\ZZ/2).$ We claim that the section can be chosen so that $\be =0,$ namely so that $\varphi([E'])=([E],0)$.

Let $\ga \in H^1(S_E;\ZZ/2)$ be the Poincar\'{e} dual of $[E'] \in H_2(S_E;\ZZ/2).$ According to the Gysin sequence \eqref{gysin}, $f(\ga) =1$ in $H^0(E;\ZZ/2)\cong \ZZ/2$. This holds since the intersection number $[E']\cdot [F] =1$ in $S_E$, where $[F] \in H_1(S_E;\ZZ/2)$ is the homology class of the circle fiber.

As previously observed, homotopy classes of sections of $S_E$ are in one-to-one correspondence with elements in $H^1(E;\ZZ_w)$. Let $r \co H^1(E;\ZZ_w)\to H^1(E;\ZZ/2)$ be the map obtained by reducing coefficients mod 2, and note that this map is surjective. The group $H^1(E;\ZZ_w)$ acts on $H^1(S_E;\ZZ/2)$ by $\xi \cdot \ga = \ga + p^*(r(\xi))$, where $\xi \in H^1(E;\ZZ_w)$ and $\ga \in H^1(S_E;\ZZ/2)$. Note that the action is transitive on the fiber $f^{-1}(1)$ of \eqref{gysin}. Thus every element $\ga \in H^1(S_E;\ZZ/2)$ with $f(\ga)=1$ is the Poincar\'{e} dual of the image $[E']$ of some section of $S_E.$ Correspondingly, every element in $\varphi_1^{-1}([E]) \subset H_2(S_E;\ZZ/2)$ is the image of some section. In particular, applied  to the element $\al \in H_2(S_E;\ZZ/2)$ with $\varphi(\al) =([E],0)$,  this shows that $\al$ can be represented by a section. This proves the claim.

Let $[E']$ be the image of a section of $S_E$ with $\varphi([E']) =([E],0)$. Then $E'$ is a compact surface embedded in $S_E$ homeomorphic to $E$. Let $\varphi \co S_E \to BO$ be the classifying map for the normal bundle of the embedding of $E'$ in $S_E$.

Since $BO=\RP^\infty = K(\ZZ/2,1)$,  the set of homotopy classes $[S_E,\RP^\infty]$ is isomorphic to $H^1(S_E;\ZZ/2)$. 
Since $S_E$ is compact, the image of $\varphi$ must lie in $\RP^N \subset \RP^\infty$ for some $N.$  
By the claim, $\varphi$ extends to a map $\varphi \co X \to \RP^N$, 
which we can choose to be transverse to the codimension one submanifold $\RP^{N-1} \subset \RP^N$. Taking $V' = \varphi^{-1}(\RP^{N-1})$,  it follows that $V'$ is a compact 3-manifold in $X$ with $\partial V' = E'.$ The manifold $V$ is obtained from $V'$ by attaching a cylinder in $N_E$ which connects $E$ and $E'$.
\end{proof}


Let $K\subset \Si \times I$ be a knot with spanning surface $F \subset \Si \times I,$ and suppose that $W$ is a compact oriented 3-manifold with boundary $\partial W = \Si$. Let $S \subset W \times I$ be an oriented surface with boundary $\partial S =K,$ and set $E= F \cup S$. 
It is a closed, unoriented surface in $W \times I.$
The hypotheses of \Cref{boundary} require that $[E]=0$ in $H_2(W\times I; \ZZ/2)$ and that it has normal Euler number $e(E)=0$.

The next proposition shows that if $W$ is a handlebody, then up to $S^*$-equivalence, $E$ always bounds.

\begin{theorem} \label{thm-surface-bound}
Let $K\subset \Si \times I$ with spanning surface $F \subset \Si \times I$. Let $W$ be a handlebody with $\partial W =\Si$ and $S \subset W \times I$ an oriented surface with $\partial S =K$. Then there exists a spanning surface $F'$ which is $S^*$-equivalent to $F$ such that $E'=F'\cup S$ bounds a 3-manifold $V\subset W\times I$.
\end{theorem}

\begin{proof}
By adding half-twisted bands, we can find a spanning surface $F'$ such that $e(F')=0$. Since $S$ is orientable, $e(S)=0$, and it follows that $e(E)=e(F')+e(S)=0$.
 
Since $W$ is a handlebody, we see that $H_2(W\times I; \ZZ/2)=0$. Thus $[E']=0$. The result now follows from \Cref{boundary}.
\end{proof}

We now show that if $K$ is slice, then we can always find a slice disk $D$ in $W\times I$ for some 3-manifold $W$ such that $E=F\cup D$ satisfies $[E]=0$ in $H_2(W\times I;\ZZ/2)$.

\begin{theorem}\label{discussion}
Let $K$ be slice, then there is a slice disk $D\subset W\times I$ for some 3-manifold $W$ such that $E=F\cup D$ bounds a 3-manifold $V\subset W\times I$.
\end{theorem}

\begin{proof}
As before, by adding half-twisted bands, we can arrange that $e(E)=0$.

Consider the long exact sequence in homology for the triple $(W \times I, \Si \times I, K)$:
\begin{equation}\label{long-exact-seq}
\xymatrix{H_2(\Si \times I, K;\ZZ/2)\ar[r]^{i_*} & H_2(W \times I, K;\ZZ/2)\ar[r]^{j_*\quad} &H_2(W \times I, \Si \times I;\ZZ/2).}
\end{equation}

The spanning surface $F$ gives an element $[F] \in H_2(\Si \times I, K;\ZZ/2),$ which maps to $[K] \in H_1(K;\ZZ/2)$ under the long exact sequence of the pair $(\Si \times I, K)$. The homology group $H_2(\Si \times I, K;\ZZ/2)$ is isomorphic to the Klein four-group $\ZZ/2 \times \ZZ/2$, and there are exactly two elements in $H_2(\Si \times I, K;\ZZ/2)$ that map to $[K]$. Both occur as the homology class of a spanning surface for $K$.

The slice disk $D$ in $W \times I$ gives an element $[D] \in H_2(W \times I, K;\ZZ/2)$ which also maps to $[K] \in H_1(K;\ZZ/2)$ under the long exact sequence of the pair $(W \times I, K)$. Thus, $E=F \cup D$ has $[E]=0$ in $H_2(W \times I;\ZZ/2)$ provided that $[D]$ pulls back under $i_*$ in \eqref{long-exact-seq} to a class in $H_2(\Si \times I, K;\ZZ/2).$

By exactness of \eqref{long-exact-seq}, we see that $[D]$ pulls back if and only if it lies in the kernel of $j_*$. Thus, the obstruction is the element $j_*([D])$ in 
$$ H_2(W \times I, \Si \times I;\ZZ/2)\cong H_2(W,\Si;\ZZ/2) \cong H^1(W;\ZZ/2),$$ 
where the first isomorphism follows from $(W \times I, \Si \times I) \simeq (W, \Si)$ and the second from Poincar\'{e} duality. 
The corresponding element $\ga \in H^1(W;\ZZ/2)$ determines a homomorphism $H_1(W;\ZZ) \to \ZZ/2$.

Let $\wt{W}$ be the associated two-fold cover of W. Then since $\pi_1(\Si)$ lies in the kernel of $\ga$, it restricts to the trivial cover on $\Si.$  Therefore, it has two copies of $\Si$ in its boundary. Attach a handlebody to one of them to get a 3-manifold with boundary $\Si$. Consider now the knot $K$ and spanning surface $F $ in the (thickening of the) boundary of $\wt{W}$. 
The disk $D$ lifts to  a slice disk $\wt{D}$ in $\wt{W} \times I$. 
Notice $[\wt{D}]$ maps to zero under the map $H_2(\wt{W} \times I, K;\ZZ/2) \to H_2(\wt{W} \times I, \Si \times I;\ZZ/2)$.

The result now follows from \Cref{boundary}.
\end{proof}

If $E=F\cup S$, with $g(S)>0$, and $\pi^{-1}(S)$ is disconnected, the same argument applies to show there is a 3-manifold $\wt{W}$ such that $\wt{E}=F\cup\wt{S}$ bounds a 3-manifold in $\wt{W}\times I$. The case which is problematic is when $\pi^{-1}(S)$ is connected.


\section{Slice obstructions} \label{sec3}
Throughout this section, $K \subset \Si \times I$ will be a $\ZZ/2$ null-homologous knot in a thickened surface. 
Recall that a knot in a thickened surface is $\ZZ/2$ null-homologous if and only if it admits a spanning surface.

Our goal in this section is to show that, for a knot $K\subset \Si \times I$ with spanning surface $F$,
if $K$ is slice and $\det(K,F)\neq 0$, then 
$\si(K,F)=0$ and $\det(K,F)$ is a perfect square. 

If $K$ is slice, then \Cref{discussion} implies that there is a 3-manifold $W$ with slice disk $D \subset W \times I$ such that the closed surface $E= F \cup D$ bounds a compact 3-manifold $V\subset W\times I$.

\begin{lemma} \label{kernel}
There is a generating set $\{\ga_1, \ldots, \ga_{2g}\}$ for $H_1(F; \ZZ)$ such that $\ga_1, \ldots, \ga_{g}$ lie in the kernel of the map $H_1(F;\ZZ) \to H_1(V;\QQ).$
\end{lemma}

\begin{proof}
Let $i\co E\to V$ be the inclusion map. First assume $V$ is orientable. Then by \cite[VI, Theorem 10.4]{Bredon} and/or \cite[Lemma  3.5]{Hatcher-2007}, $\,{\rm dim}\,(\,{\rm Ker}\,i_{*})=\frac{1}{2} \,{\rm dim}\,(H_1(E;\QQ))$. By a Mayer-Vietoris argument, $H_1(F; \ZZ)\cong H_1(E; \ZZ)$, and by the Universal Coefficient Theorem, $H_1(E; \ZZ)\otimes \QQ\cong H_1(E; \QQ)$, and the conclusion follows in this case.

Now assume $V$ is not orientable.
If $i_{*}\co H_1(E;\ZZ/2)\to H_1(V;\ZZ/2)$, then $\,{\rm dim}\,(\,{\rm Ker}\,i_{*}) =\frac{1}{2} \,{\rm dim}\,(H_1(E;\ZZ/2))$. We can choose a generating set $\{\ga_1, \ldots, \ga_{2g}\}$ for $H_1(F; \ZZ)$, such that the mod 2 reduction of $\{\ga_1, \ldots, \ga_{g}\}$ generates $\,{\rm Ker}\,i_{*}$, so they must all be mapped to torsion elements in $H_1(V;\ZZ)$, and therefore, they all lie in the kernel of the map $H_1(F;\ZZ) \to H_1(V;\QQ).$
\end{proof}

Let $Q$ be a symmetric bilinear form  on a vector space $V$ over $\RR$. A subspace $U\subset V$ is called \textit{totally isotropic} if $Q$ vanishes on $U$. If $U$ is a totally isotropic subspace of maximal dimension, then
\begin{equation} \label{max_total_isot}
2 \dim U = r-|\si|+n,
\end{equation}
where $r, \si$ and $n$ are the rank, signature, and nullity of $Q$.

\begin{theorem}\label{slice}
Let $K \subset \Si \times I$ be a knot with spanning surface $F$ such that $e(F)=0$. Suppose further that $K$ is slice. Then
$$|\si(K,F)|\leq n(K,F).$$ 
In particular, if $n(K,F)=0$, then $\si(K,F)=0$.
\end{theorem}

\begin{proof}
By \Cref{discussion} there is a slice disk $D\subset W\times I$ for some 3-manifold $W$ such that $E=F\cup D$ bounds a 3-manifold $V\subset W\times I$.
Now by \Cref{kernel}, we have a set $\{ \ga_1, \ldots, \ga_{2g}\}$ of generators for $H_1(F;\ZZ)$ such that $\ga_1, \ldots, \ga_{g}$ lie in the kernel of the map $H_1(F;\ZZ) \to H_1(V;\QQ).$ Let $U$ be the subgroup of $H_1(F;\ZZ)$ generated by $\ga_1, \ldots, \ga_{g}$. For any $a\in U$, there is a surface $A$ with boundary a multiple of $a$. Now suppose $b\in U$, and $B$ is a surface with boundary a multiple of $b$. We have
$$\cG_{F}(a,b)=\tfrac{1}{2}(\lk(\tau a,b)+\lk(\tau b,a)).$$
Let $N(V)$ be a tubular neighborhood of $V$ in $W \times I$. Then $N(V)$ is a (possibly twisted) $I$-bundle over $V$ with boundary $\partial N(V)$ a $\{\pm 1\}$-bundle over $V$. If $p\co N(V) \to V$ is the projection map, then $p^{-1}(A) \cap \partial N(V)$ gives a surface $\wt{A}$ in $W \times I$ with boundary $\tau a$. (In fact, $p|_{\wt{A}}\co  \wt{A} \to A$ is a two-fold covering.) Notice that $\wt{A}$ and $B$ are disjoint, hence the intersection $\wt{A} \cdot B =0.$ Thus,  $\lk(\tau a,b) = \wt{A} \cdot B = 0$. A similar argument shows that $\lk(\tau b,a)=0$. Thus $\cG_{F}(a,b)=0$ for $a,b \in U.$ Hence $U$ is a totally isotropic subspace for the Gordon-Litherland pairing. Equation \eqref{max_total_isot} gives the inequality:
$$2\dim U \leq \dim H_1(F;\ZZ)-|\sig (\cG_{F})|+n(K,F),$$ but since $\dim U = \tfrac{1}{2} \dim H_1(F;\ZZ)$ in our case, it follows that  $|\sig (\cG_{F})|\leq n(K,F).$

Now the self-intersection number of $E$, which is the normal Euler number of $E$, vanishes. To calculate
the self-intersection number of $E$, we push a copy of $E$ off itself.
Since the slice disk $D$ is orientable, we can push it off so there are no self-intersection points over the slice disk. Therefore, the intersection number is
$-\lk(K,K')=e(F,K)=0$. It follows that $|\sig\cG_F|=|\si(K,F)|\leq n(K,F).$
\end{proof}

\begin{remark} \label{remark-tube}
If $L \subset \Si \times I$ is a link with spanning surface $F \subset \Si \times I$, we can  construct a new surface $F'=F\#_\tau \Si$ by connecting $F$ to a parallel copy of $\Si$ near $\Si \times \{0\}$ by a small thin tube $\tau$.
Notice that if $F'=F\#_\tau \Si$, then $e(F')=e(F)$ and $[F'\cup D]=[F\cup D]$ in $H_2(W\times I;\ZZ/2)$. In particular, if $F\cup D$ satisfies the hypothesis of  \Cref{slice}, then $F'\cup D$ does as well.
\end{remark}

\begin{theorem}\label{determinant}
Let $K \subset \Si \times I$ be a knot with spanning surface $F$ such that $e(F)=0$. Suppose further that $K$ is slice. Then  $\det (K,F)$ is
a perfect square.
\end{theorem}

\begin{proof}
Choose a slice disk $D\subset W\times I$ as in the proof of \Cref{slice}.
Let $\{\ga_1, \ldots, \ga_{2g}\}$ be a basis for $H_1(F; \ZZ)$, obtained by \Cref{kernel}. Then by \Cref{slice}, the matrix of $\cG_{F}$ with respect to this basis has the form
$\begin{bmatrix} 0&A \\ A^\tr & B  \end{bmatrix}$, where $A,B$ are $g\times g$ matrices. Then
\begin{equation*}
\begin{split}
\begin{bmatrix} I& 0 \\ -A^\tr & xI  \end{bmatrix}
\begin{bmatrix} xI& A \\ A^\tr & B  \end{bmatrix}&=
\begin{bmatrix} xI& A \\ 0 & -A^\tr A+xB  \end{bmatrix},\\
x^g \det \begin{bmatrix} xI& A \\ A^\tr & B  \end{bmatrix}&=
x^g \det (-A^\tr A+xB).
\end{split}
\end{equation*}
We divide both sides by $x^g$, and in the remaining equation put $x=0$, it follows that
$$\det \begin{bmatrix} 0&A \\ A^\tr & B  \end{bmatrix}=(-1)^g(\det A)^2,$$
therefore, $\det (K,F)$ is a perfect square.
\end{proof}
\begin{remark}
For a null-homologous knot $K \subset \Si \times I$ with Seifert surface $F\subset \Si \times I$, 
one can show that $\det(K,F)= \nabla^{+}_{K,F}(-1)=\nabla^{-}_{K,F}(-1)$. Here $\nabla^{\pm}_{K,F}(t)$ refers to the directed Alexander polynomial defined in \S 2.4 of \cite{Boden-Chrisman-Gaudreau-2020}. Thus, for these knots, \Cref{determinant} can be deduced from Theorem 2.8 of\cite{Boden-Chrisman-Gaudreau-2020}, giving the Fox-Milnor condition that $\nabla^{\pm}_{K,F}(t) = f^{\pm}(t)f^{\pm}(t^{-1})$ for some $f^\pm(t) \in \ZZ[t]$ when $K$ is slice. 
\end{remark}

\begin{theorem} \label{thm-sign}
Let $K \subset \Si \times I$ be a  knot with spanning surface $F$  such that $e(F)=0$. Suppose further that $S \subset W \times I$ is an orientable surface with genus $g(S)$  and boundary $\partial S=K$. Assume that the closed surface $E= F \cup S$ satisfies $[E]=0$ in $H_2(W\times I; \ZZ/2).$  Then $$|\si(K,F)| \leq 2g(S)+n(K,F).$$ 
In particular, if $n(K,F)=0$, then $\si(K,F)\leq 2 g(S)$.
\end{theorem}

\begin{proof}
Let $k= g(S)$. Since $E=F \cup S$ bounds a 3-manifold, it follows that $H_1(F; \ZZ)$ has even rank, say $2h$. Furthermore, since $S$ is orientable, $F$ has normal Euler number $e(F)=e(E)=0.$
In the case $k \geq h$, then the conclusion follows trivially from the fact that $|\si(K,F)|=|\sig \cG_F|\leq 2h.$ Therefore, we can assume that $k<h.$

We can repeat the proof of \Cref{kernel}, to show that there is a generating set $\{\ga_1, \ldots, \ga_{2h}\}$ for $H_1(F; \ZZ)$
such that $\ga_1, \ldots, \ga_{h-k}$ lie in the kernel of the map $H_1(F;\ZZ) \to H_1(V;\QQ).$

For a symmetric bilinear form with rank $r$, signature $\si$, and nullity $n$, by \Cref{max_total_isot}, any totally isotropic subspace has rank at most  $(r-|\si|+n)/2$. Hence
$$h-k\leq (2h-|\sig \cG_{F}|+n(K,F))/2,$$
and the conclusion follows.
\end{proof}

\begin{remark}
\Cref{thm-sign} is easiest to apply in case $W$ is a handlebody. 

Let $K \subset \Si \times I$ be a knot with spanning surface $F$, and $S \subset  W\times I$ is an orientable surface of genus $g$. If $W$ is a handlebody, then by \Cref{thm-surface-bound}, there is a spanning surface $F'$ that is $S^*$-equivalent to $F$ such that $E' = F' \cup S$  bounds a 3-manifold $V$. In particular, since $F$ and $F'$ are $S^*$-equivalent, $\si(K,F) =\si(K,F)$ and $n(K,F) =n(K,F')$. Applying \Cref{thm-sign} to $F'$, it follows that 
$$|\si(K,F)| \leq 2g(S) + n(K,F).$$
\end{remark}

In order to state the next result, we need to review some terminology.

\begin{definition} \label{defn:123}
Let $D$ be a link diagram on $\Si$. 
\begin{itemize}
\item[(i)] $D$ is \textit{alternating} if its crossings alternate between over- and under-crossings along every component of $D$, 
\item[(ii)] $D$ is \textit{cellularly embedded} if $\Si \sm D$ is a union of disks, and 
\item[(iii)] $D$ is  \textit{checkerboard colorable} if the regions of $\Si \sm D$ can be colored by two colors so that no two regions sharing an edge have the same color. 
\end{itemize}
\end{definition}

Properties (ii) and (iii) of \Cref{defn:123} are clearly unaffected by crossing changes on $D$. Kamada showed that a cellularly embedded diagram is checkerboard colorable if and only if it can be made alternating by crossing changes (see \cite[Lemma 7]{Kamada-02}).

\begin{theorem} 
Let $K \subset \Si \times I$ be a knot in a thickened surface represented by a cellularly embedded, alternating diagram. Let $F$ and $F^*$ be the two checkerboard surfaces associated to the coloring. Then $\det(K,F)\neq 0$ and $\det(K,F^*)\neq 0$. Further, $|\si(K,F)-\si(K,F^*)| =2g,$ where $g$ is the genus of $\Si.$ In particular, if $\Si$ has genus $g\geq 1,$ then $K$ is not slice.
\end{theorem}

\begin{proof}
Using Lemma 6 
of \cite{Boden-Karimi-2020}, we see that $K$ admits a positive definite and negative definite spanning surface. Since there are only two $S^*$-equivalence classes of surfaces, it follows that one of the definite spanning surfaces is $S^*$-equivalent to $F$ and the other is $S^*$-equivalent to $F^*$. In particular, since $\det(K,F)$ depends only on the $S^*$-equivalence class of $F$, it follows that $\det(K,F)\neq 0$ and $\det(K,F^*)\neq 0$.

Arguing as in the proof of Theorem 19 
of \cite{Boden-Karimi-2020}, we see that $|\si(K,F)-\si(K,F^*)| =2g$. We then apply \Cref{cor-slice} to conclude that $K$ is not slice.
\end{proof}


\section{Brown Invariants} \label{sec4} 
In this section, we recall the Brown invariant of a proper quadratic enhancement of a finite dimensional $\ZZ/2$ inner product space, following \cite{Brown, Kharlamov-Viro}. Using it, one can define invariants for $\ZZ/2$ null-homologous links $L\subset \Si \times I$ together with a choice of spanning surface $F\subset \Si \times I$. The invariants are called Brown invariants; they are denoted $\Br_F(L)$ and take values in $\ZZ/8 \cup \{ \infty \}$. They have been studied previously by several different authors \cite{Matsumoto, Kirby-Taylor, Gilmer-1993, Kirby-Melvin, Klug-2020}.  

We outline two methods for defining and computing the Brown invariants of links in thickened surfaces. The first involves the Gordon-Litherland pairing and the second Goeritz matrices. We prove a duality theorem that relates the two approaches. We also prove a vanishing result for $\Br_F(K)$ for $\ZZ/2$ null-homologous knots $K \subset \Si \times I$ which are slice. The Brown invariants are seen to change sign under taking vertical or horizontal mirror image,
and they specialize to the Arf invariants when the link is $\ZZ$ null-homologous.  
\subsection{Quadratic enhancements and the Brown invariant} \label{section-Brown}
Let $V$ be a finite-dimensional vector space over $\ZZ/2$ and with a possibly singular symmetric bilinear form $\bigcdot.$ Even though it may be singular, we refer to $\bigcdot$ as an inner product and call $(V,\, \bigcdot\,)$ an inner product space.

A \textit{$\ZZ/4$-valued quadratic enhancement} of $(V,\bigcdot\,)$ is a map $\varphi \co V\to\ZZ/4$ such that
\begin{equation} \label{eqn:enhancement}
\varphi(u+v)=\varphi(u)+\varphi(v)+j(u\bigcdot v),
\end{equation}
for all $u,v \in V$, where $j\co\ZZ/2\to \ZZ/4$ is the monomorphism with $j(1) \equiv 2$ (mod 4). We call $(V,\bigcdot\, ,\varphi)$ a $\ZZ/4$ \textit{enhanced space}. 

\Cref{eqn:enhancement} implies that $\varphi(0)=0$ and that $\varphi(v) \equiv v\bigcdot v $ (mod 2) for all $v \in V$. The inner product is said to be \textit{even} if $v\bigcdot v=0$ for all $v \in V.$ In that case, \eqref{eqn:enhancement} implies that $\varphi(v)$ is even for all $v \in V$. When $\varphi$ is even, we can define an ordinary $\ZZ/2$-valued quadratic enhancement $q \co V \to \ZZ/2$ by setting $q(v)= \varphi(v)/2$ for all $v \in V.$ By \eqref{eqn:enhancement}, the ordinary quadratic form $q$ satisfies $q(u+v)=q(u)+q(v)+u\bigcdot v$ for all $u,v \in V.$ 

The \textit{radical} of $(V,\,\bigcdot\,)$ is the subspace
$$R=\{u\in V \mid u\bigcdot v=0,\ \text{for\ every}\ v\in V\}.$$
The form $\varphi$ is said to be \textit{non-singular} if $R=0$, and \textit{proper} if $\varphi|_R =0.$

Given a non-singular enhanced space $(V,\bigcdot,\varphi)$, consider the \textit{Monsky sum} 
\begin{equation} \label{eqn:Monsky}
\la(\varphi)=\sum_{v \in V}  i^{\varphi(v)}.
\end{equation}
By \cite[Lemma 3.2]{Brown}, it follows that $\la$ is multiplicative under orthogonal sum (i.e., $\la(\varphi_1 \oplus \varphi_2) = \la(\varphi_1) \la(\varphi_2)$), and by  \cite[Lemma 3.3]{Brown}, we see that $\la(\varphi)^8$ is positive and real provided that $\varphi$ is non-singular. Therefore, it follows that  
\begin{equation} \label{eqn:Brown}
\la(\varphi) =  (\sqrt{2} )^{\dim V} e^{\Br(\varphi) {\pi i} /4}
\end{equation} 
for some integer $\Br(\varphi)$ well-defined modulo 8. The element  $\Br(\varphi) \in \ZZ/8$ is called the \textit{Brown invariant}. 

If $\varphi$ is singular and proper, then it induces a non-singular quadratic enhancement $\wbar{\varphi}$ on $V/R$, and we define $\Br(\varphi)=\Br(\wbar{\varphi}).$ If  $\varphi$ is not proper, then we set $\Br(\varphi)=\infty$ and say that the Brown invariant is undefined. 

Observe that the Monsky sum \eqref{eqn:Monsky} is well-defined even when $\varphi$ is singular. In fact, if $\varphi$ is not proper, then $\la(\varphi) =0$. If $\varphi$ is singular and proper, then
\begin{equation} \label{eqn:Monsky2}
\la(\varphi) =  (\sqrt{2} )^{(\dim V+\dim R)} e^{\Br(\varphi) {\pi i} /4}.
\end{equation}  
In particular, $\varphi$ is proper if and only if $\la(\varphi) \neq 0.$

For a fixed inner product space $(V,\,\bigcdot \,),$ the map
$$\Br\co\{\text{enhanced\ spaces}\}\lto\ZZ/8\cup\{\infty\},$$
provides a complete isomorphism invariant of quadratic enhancements of $(V,\,\bigcdot \,)$, (see Section 3.1 in \cite{Kirby-Melvin}).

If $\varphi\co V \to \ZZ/4$ is even, then the Monsky sum lies on the real line. If $\varphi$ is even and proper, then $\la(\varphi)=\pm(\sqrt{2})^{(\dim V + \dim R)}$ and $\Br(\varphi) \in \{0,4\}$. In particular, $\Br(\varphi)=4\Arf(q),$ the (usual) Arf invariant of the ordinary quadratic form $q$ associated to $\varphi$. Thus, the Brown invariant specializes to the Arf invariant when $\varphi$ is even. 

\subsection{Graphical representation of quadratic enhancements}

In \cite{Gilmer-1993}, Gilmer uses a simple $\ZZ/4$-weighted graph to represent
$(V,\bigcdot\, ,\varphi)$. Let $\{e_1,\ldots, e_n\}$ be a basis for $V$. 
The graph has a vertex for each basis element $e_i$ and an edge between two vertices $e_i$ and $e_j$ ($i \neq j$) whenever $e_i \bigcdot e_j =1$.  The vertex $e_i$ has weight $\varphi(e_i) \in \ZZ/4$.

By \eqref{eqn:enhancement}, $\varphi$ is completely determined by the values $e_i \bigcdot e_j \in \ZZ/2$ and $\varphi(e_i) \in \ZZ/4$ for $i,j=1,\ldots, n.$ Thus, the simple weighted graph determines the form $\varphi$.

For example, the four indecomposable spaces denoted  $P_\pm, T_0, T_4$ in \cite[\S 4]{Matsumoto}
and \cite[\S 3.2]{Kirby-Melvin} are represented by the weighted graphs:
$P_\pm = \bullet^{\pm 1}$, $T_0 = \!\!
\begin{tikzpicture}
\draw (0.2,0) -- (0.8,0);
\filldraw (0.2,0) circle (2pt); 
\filldraw (0.8,0) circle (2pt); 
\draw[color=black] (0,0.12) node {\scriptsize $0$};
\draw[color=black] (1.0,0.12) node {\scriptsize  $0$};
\end{tikzpicture}$
and $ T_4= \!\!
\begin{tikzpicture}
\draw (0.2,0) -- (0.8,0);
\filldraw (0.2,0) circle (2pt); 
\filldraw (0.8,0) circle (2pt); 
\draw[color=black] (0,0.12) node {\scriptsize $2$};
\draw[color=black] (1.0,0.12) node {\scriptsize  $2$};
\end{tikzpicture}$.
Using this and the Monsky sums, it is elementary to show that
\begin{equation} \label{eqn:indecomp}
\Br(P_\pm)= \pm 1, \; \Br(T_0) = 0, \; \text{and} \; \Br(T_4) = 4. 
\end{equation}

Given a basis $\{e_1,\ldots, e_n\}$, one can also represent $(V,\bigcdot\, ,\varphi)$ using a special kind of matrix. It is a symmetric  $n \times n$ matrix with  diagonal entries $\varphi(e_1),\ldots, \varphi(e_n)$ and off-diagonal entries
$e_i \bigcdot e_j$ for $i \neq j$. Basically, it is just the adjacency matrix of the simple graph
with diagonal entries given by the weights. The matrix is special in that its diagonal entries lie in $\ZZ/4$ whereas its off-diagonal entries lie in $\ZZ/2$. We call a symmetric matrix obtained this way the \textit{special matrix representative} for $\varphi$ in the basis $\{e_1,\ldots, e_n\}.$

For example, the indecomposable spaces have special matrix representatives 
$$P_\pm = [\pm 1], \;
T_0 = \begin{bmatrix} 0 & 1 \\ * & 0 \end{bmatrix}, \;
\text{ and } \; T_4 = \begin{bmatrix} 2 & 1 \\ * & 2 \end{bmatrix}.$$

The Brown invariants are additive under orthogonal sum. Orthogonal sum is represented by disjoint union of the graphs or by diagonal block sum of special matrices. Every non-singular enhanced space is the orthogonal sum of these four indecomposable spaces. Thus, \eqref{eqn:indecomp} can be used to effectively compute the Brown invariant on any non-singular enhanced space. 

In \cite{Gilmer-1993}, Gilmer describes a graphical calculus for performing such computations. His method extends to singular forms and determines when the form is proper. For instance, the 1-dimensional form represented by $\bullet^2$ has Monsky sum $1+i^2=0$, therefore, it is not proper.
Likewise,
the 2-dimensional form represented by the barbell graph  
\begin{tikzpicture}
\draw (0.2,0) -- (0.8,0);
\filldraw (0.2,0) circle (2pt); 
\filldraw (0.8,0) circle (2pt); 
\draw[color=black] (0,0.12) node {\scriptsize $k$};
\draw[color=black] (1.0,0.12) node {\scriptsize  $\ell$};
\end{tikzpicture}
has Monsky sum 
$\la = 1+i^k + i^\ell + i^{k+\ell+2}$.
It is proper if and only if $\la \neq 0$.
In particular, \!\! \begin{tikzpicture}
\draw (0.2,0) -- (0.8,0);
\filldraw (0.2,0) circle (2pt); 
\filldraw (0.8,0) circle (2pt); 
\draw[color=black] (-0.1,0.12) node {\scriptsize $-1$};
\draw[color=black] (1.0,0.12) node {\scriptsize  $1$};
\end{tikzpicture} represents an improper form.

\subsection{Brown invariants of links via the Gordon-Litherland pairing}\label{section-link}
In this section, we give the first definition of the Brown invariants for oriented links in thickened surfaces.

We start with some basic results about quadratic enhancements in terms of their special matrix representatives.

Let $U$ be a free abelian group, and let $B \co U \times U \to \ZZ$ be a symmetric bilinear pairing.
Let $V$ be the mod 2 reduction of $U$. For $[u],[v]\in V$, define $[u]\bigcdot [v]\equiv B(u,v)$ (mod 2), where the $u,v\in U$ have mod 2 reductions $[u],[v],$ respectively. Define $\varphi \co V \to \ZZ/4$ by setting $\varphi([u])\equiv B(u,u)$ (mod 4).

\begin{lemma}\label{well-defined}
The map $\varphi$ is well-defined and $(V,\bigcdot\, ,\varphi)$ is a $\ZZ/4$ enhanced space. 
\end{lemma}

\begin{proof}
If $u'\in U$ is another element with $[u']=[u]$, then $u'=u+2v$ for some $v\in U$. It follows that
\begin{equation*}
\begin{split}
\varphi([u'])&\equiv B(u+2v,u+2v) \text{ (mod $4$)},\\
&=B(u,u)+ B(u,2v)+B(2v,u)+B(2v,2v),\\
               &= B(u,u)+4B(u,v)+ 4B(v,v),\\
               &\equiv \varphi([u]) \text{ (mod $4$)}.
\end{split}
\end{equation*}
This shows that $\varphi$ is well-defined. 

Next, we show that $\varphi$ is a $\ZZ/4$ quadratic form $(V,\bigcdot\,)$.
If $[u],[v]\in V$, let $u,v \in U$ be chosen so that they map to $[u],[v]$ under reduction mod 2. Then
\begin{equation*}
\begin{split}
\varphi([u]+[v]) &\equiv B(u+v,u+v) \text{ (mod $4$),}\\
&= B(u,u)+B(v,v) +2B(u,v),\\
&\equiv \varphi([u])+\varphi([v])+2(u\bigcdot v) \text{ (mod $4$)}.
\end{split}
\end{equation*}
The claim now follows.
\end{proof}

Given a symmetric integral $n\times n$ matrix $M$, 
let $\wbar{M}$ be the matrix obtained from $M$ by reducing the diagonal entries mod 4 and the off-diagonal entries mod 2. 

If $U$ is a free abelian group with basis $\{e_1,\ldots, e_n\}$, then the matrix $M$ determines a bilinear pairing $B\co U \times U \to \ZZ$ by setting $B(e_i,e_j) = M_{ij}$.    

Let $V$ be the mod 2 reduction of $U$, and let $\bigcdot$ be the inner product on $V$ determined by $B$.
Let $\varphi_M$ denote the $\ZZ/4$ quadratic form on $(V, \bigcdot\, )$ defined by $B$ as above.
Notice that the special matrix representative of $\varphi_M$ is given by $\wbar{M}$.  

\begin{lemma}\label{unimodular}
If $M,M'$ are two symmetric integral $n\times n$ matrices that are unimodular congruent, then $\Br(\varphi_{M})=\Br(\varphi_{M'}).$
\end{lemma}

\begin{proof}
Fix a basis $\{e_1,\ldots, e_n\}$ for $U$. For an element $u\in U$, by abuse of notation, denote the coordinate of $u$ with respect to the basis (an $n\times 1$ column) also by $u$.

Let $B,B'$ be the symmetric bilinear forms on $U$ determined by $M,M',$ respectively. 
Since they are unimodular congruent, we have $M'=P^{\tr}MP$ for some unimodular matrix $P$. 
Then
$$B'(u,v)=u^{\tr}M' v=u^{\tr}P^{\tr}M P v=B( Pu,Pv).$$
It follows that
$$\varphi_{M'}([u])= \varphi_{M}([Pu])=\varphi_{M}(\Phi[u]), $$
where $\Phi\co V\to V$ is given by $\Phi([u])=[Pu]$. Since $P$ is unimodular, $\Phi$ is an isomorphism. Thus $\varphi_{M'}=\varphi_{M} \circ \Phi$, and by \cite[Definition 1.19]{Brown}, we see that $\varphi_{M}$ and $\varphi_{M'}$ are isomorphic. The result now follows from \cite[Theorem 1.20, (i)]{Brown}.
\end{proof}

Let $L\subset \Si\times I$ be an unoriented link and $F \subset \Si \times I$ a spanning surface for $L$.
Take $V_F=H_1(F;\ZZ/2)$ with  symmetric bilinear form given by the mod 2 Gordon-Litherland pairing.

Define $\varphi_F \co V_F \to \ZZ/4$ by setting $\varphi_F([a])=\cG_{F}(a,a)$.  Then by \Cref{well-defined}, it follows that $\varphi_F$ is well-defined and 
that $(V_F, \cG_F, \varphi_F)$ is a $\ZZ/4$ enhanced space.

We denote by $\Br(\varphi_F)$ the Brown invariant of $(V_F, \cG_F,\varphi_F)$. 

\begin{definition}\label{defn:Brown-link}
The Brown invariant of an oriented link $L \subset \Si \times I$ with spanning surface $F \subset \Si \times I$ is denoted $\Br_F(L)$ and defined by the formula 
$$\Br_F(L)=\Br(\varphi_F)+\tfrac{1}{2}e(F,L).$$
\end{definition}
Note that $\Br_F(L) \in \ZZ/8 \cup \{\infty\}$, and it depends on the choice of spanning surface $F$. It depends on the orientation of $L$ only through the Euler class $e(F,L) = e(F) -\la(L)$.
The next result shows that $\Br_F(L)$ is invariant under $S^*$-equivalence of $F$.

\begin{proposition} \label{prop:klug}
Let $L\subset \Si\times I$ be a link with spanning surface $F$. Then $\Br_F(L)$ depends only on the $S^*$-equivalence class of $F$. 
\end{proposition}

\noindent For a proof, see Section 6 of \cite{Klug-2020}.

\begin{figure}[htbp]
\begin{center}
\includegraphics[scale=0.90]{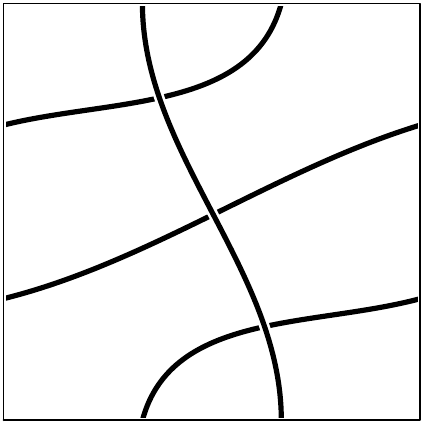}\hspace{.3cm}
\includegraphics[scale=0.90]{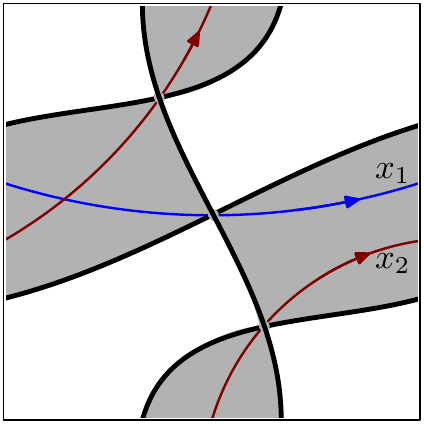}\hspace{.3cm}
\includegraphics[scale=0.90]{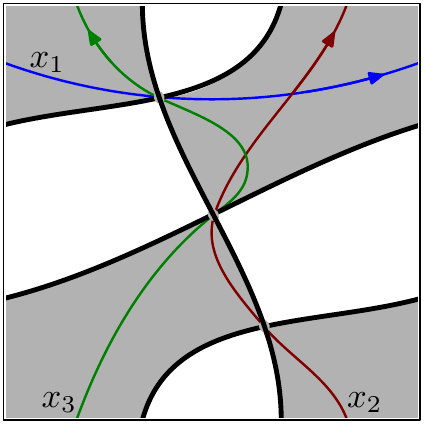}\hspace{.3cm}
\caption{\small A knot in the thickened torus, and two spanning surfaces for it.}
\label{fig-3-5-1}
\end{center}
\end{figure}

\begin{example}\label{ex-3-5-1}
\Cref{fig-3-5-1} shows a knot in the thickened torus with two spanning surfaces $F$ (middle) and $F^*$ (right). 

For the first surface, using the basis $\{x_1,x_2\}$ for $H_1(F;\ZZ)$, we compute that the Gordon-Litherland pairing $\cG_F$ is given by the diagonal matrix $[1]\oplus [-2]$.
Thus, we can represent $\varphi_F$ by the special matrix $[1]\oplus [2]$.
As previously noted, the form represented by $[2]$ is not proper, nor is any form containing $[2]$ as an orthogonal summand. 
Therefore,  $\Br(\varphi_F)=\infty,$ and $\Br_{F}(K)=\infty$ is undefined. 

For the second surface, using the basis $\{x_1,x_2,x_3\}$ for $H_1(F^*;\ZZ)$, we compute the Gordon-Litherland pairing $\cG_{F^*}$ and see that 
$\varphi_{F^*}$ is represented by the special matrix 
$$\begin{bmatrix} 1 & 0 & 1 \\  * & 0 & 1 \\ * & * & 0\end{bmatrix}.$$ 
A straightforward calculation reveals that 
$\la(\varphi_{F^*})= 2+2i =(\sqrt{2})^3e^{ {\pi i} /4}. $
Therefore, $\Br(\varphi_{F^*})=1$. Since $e(F^{*},K)=2$, it follows that $\Br_{F^*}(K)= 1+\frac{1}{2}e(F^{*},K)=2$.
\hfill $\Diamond$ \end{example}


\subsection{Brown invariants of links via Goeritz matrices}
In this section, we give a second way to define Brown invariants for
checkerboard colored link diagrams in a thickened surface. 

Let $D \subset \Si$ be a diagram for a checkerboard colorable link $L$, and let $\xi$ be a checkerboard coloring for the regions $\Si \sm D.$ We use $F_\xi$ to denote the checkerboard surface obtained from the black regions, so $F_\xi$ is a union of disks and bands, with one disk for each black region of $\Si \sm D$ and one half-twisted band for each crossing of $D$.

The \emph{Tait graph} is the associated graph in $\Si$; it has one vertex for each black region and one edge for each crossing of $D$. If $C_D$ is the set of all crossings of $D$, we define the incidence number $\eta_c = \pm 1$ for $c \in C_D$ according to \Cref{fig-crossing-eta}. We also label each of the edges of the Tait graph with the sign of its incidence number.

\begin{figure}[htbp]
\begin{minipage}{0.48\textwidth}
     \centering
     \includegraphics[scale=0.90]{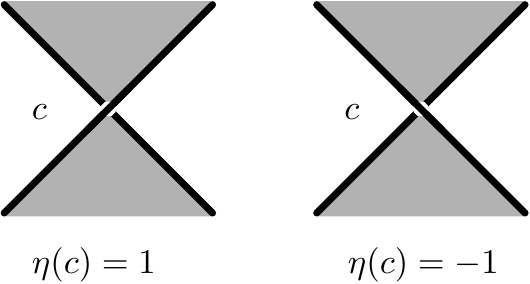}
     \vspace*{-0.2cm}
     \caption{\small Incidence number}\label{fig-crossing-eta}
     
   \end{minipage}
\begin{minipage}{0.48\textwidth}
     \centering
     \includegraphics[scale=0.90]{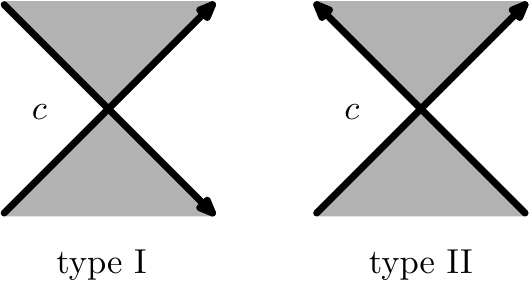}
      \vspace*{-0.2cm}
      \caption{\small Crossing type}\label{fig-crossing-type}
   \end{minipage}
\end{figure}

Let $X_0,X_1,\ldots, X_n$ be a numbering of the white regions  of $\Si \sm D$. Define an $(n+1) \times (n+1)$ matrix $G'_{\xi}(D)=(g_{ij})_{i,j=0,\ldots, n}$ by setting
$$g_{ij} = 
\begin{cases} 
-\sum  \eta_c & \text{ if $i \neq j$,} \\
-\sum_{k\neq i} {g_{ik}} & \text{ if $i = j$.}
\end{cases}$$
The first sum is taken over all crossings $c \in C_D$ incident to both $X_i$ and $X_j$. Then  $G'_{\xi}(D)$ is a symmetric matrix with integer entries and with $\det G'_{\xi}(D) =0.$

\begin{definition} \label{GL-pairing}
The Goeritz matrix $G_\xi(D)$ is the $n \times n$ matrix obtained by deleting the first row and column from $G'_{\xi}(D).$ In other words, $G_\xi(D) =(g_{ij})_{i,j=1,\ldots, n}.$  
\end{definition}

The Goeritz matrix $G_\xi(D)$ depends on the choice of checkerboard coloring $\xi$ and on the order of the white regions, as well as on the diagram $D$ used to represent the given link $L \subset \Si \times I$. In \cite{Im-Lee-Lee-2010}, Im, Lee and Lee show how to derive link invariants (signature, nullity, and determinant) from $G_\xi(D)$ for non-split checkerboard colorable links in thickened surfaces. Their invariants depend on the choice of checkerboard coloring, and they get two sets of invariants, one for each coloring.

In a similar way, one can define Brown invariants associated to the Goeritz matrix. 
Since $G_{\xi}(D)$ is  a symmetric integral matrix, we can define $\varphi_\xi \co V \to \ZZ/4$ to be the quadratic enhancement associated to it as in \Cref{unimodular}.
Here, we take $V = \cK_F = \Ker(H_1(F;\ZZ/2)\to H_1(\Si;\ZZ/2))$, where $F$ is the black checkerboard surface.
 
Let $\Br(\varphi_\xi)$ denote the Brown invariant associated to $\varphi_\xi$. Define the correction term
$$\mu_{\xi}(D) = \sum_{c \text{ type II}} \eta_c.$$
(Note that $\mu_\xi(D) = -\frac{1}{2} e(F,L)$ by Lemma 2.4 \cite{Boden-Chrisman-Karimi-2021}.) 
\begin{definition} \label{defn:Brown-link-Goeritz}
The Brown invariant of the oriented checkerboard colorable link diagram $D$ with coloring $\xi$ is given by
$$\Br_{\xi}(D)=\Br(\varphi_{\xi})-\mu_{\xi}(D).$$
\end{definition} 

The quantity $\Br_{\xi}(D)$ is an element of $\ZZ/8 \cup \{\infty\}$, and it depends on a choice of checkerboard coloring $\xi$. It  depends on the orientation of $D$ only through the correction term $\mu_{\xi}(D).$

At this point, one could show directly that $\Br_{\xi}(D)$ is invariant under the Reidemeister moves by mimicking the arguments of \cite{Im-Lee-Lee-2010} showing invariance of the signature, determinant, and nullity. We proceed with an indirect proof that relates the Brown invariants of \Cref{defn:Brown-link} to those of \Cref{defn:Brown-link-Goeritz}. By \Cref{prop:klug}, the former invariants are invariant under $S^*$-equivalence, therefore, it follows that $\Br_{\xi}(D)$ gives a well-defined link invariant depending only on the choice of checkerboard coloring.

\begin{figure}[htbp]
\begin{center}
\includegraphics[scale=0.90]{Figures/3-5-in-torus.pdf}\hspace{.3cm}
\includegraphics[scale=0.90]{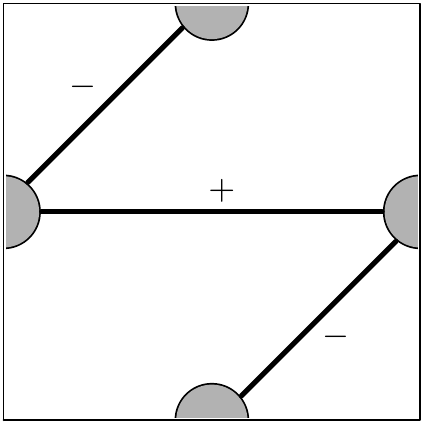}\hspace{.3cm}
\includegraphics[scale=0.90]{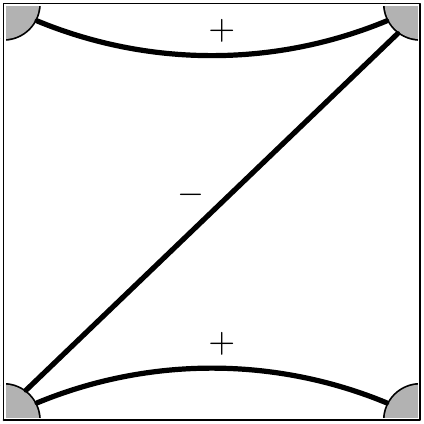}
\caption{\small A knot in the thickened torus and the two Tait graphs for its checkerboard surfaces.}
\label{fig-3-5}
\end{center}
\end{figure}

\begin{example}\label{ex-3-5-2}
\Cref{fig-3-5} shows the same knot as in the \Cref{ex-3-5-1}, along with the Tait graphs for $F$ (middle) and for $F^*$ (right). 

For the surface $F$, the associated Goeritz matrix is empty. Hence, $\Br(\varphi_{\xi})=0$. We have $\mu_{\xi}=-2$, and $\Br_{\xi}(K)=0-(-2)=2$.

For the surface $F^*$, the associated Goeritz matrix is $G_{\xi^*}(D)=[2]$, which we have seen represents an improper form. Thus, $\Br(\varphi_{\xi^*})=\infty$, and so $\Br_{\xi^*}(K)=\infty$ is undefined.
\hfill $\Diamond$ \end{example}


\subsection{Duality for Brown invariants}
In this section, we prove the chromatic duality theorems, which  are the analogues for Brown invariants of Theorems 4.1 and 5.4 of \cite{Boden-Chrisman-Karimi-2021}. We apply them to relate the two families of  Brown invariants. 

In \cite{Boden-Chrisman-Karimi-2021}, a similar approach is used to relate the invariants (signature, determinant, and nullity) defined in terms of the Gordon-Litherland pairing to the invariants defined by Im, Lee and Lee in terms of Goeritz matrices \cite{Im-Lee-Lee-2010}. The correspondence requires switching the checkerboard coloring, which is an aspect which is unique to this setting and still somewhat curious.

Recall from \Cref{remark-tube} that for a link $L \subset \Si \times I$ with spanning surfac $F \subset \Si \times I$, we can construct a new surface  $F'=F\#_\tau \Si$ by attaching a parallel copy of $\Si$ to $F$ using by a small tube $\tau$. From Section 4 in \cite{Boden-Chrisman-Karimi-2021}, we know $F'$ and $F$ are not $S^*$ equivalent. 
Furthermore, $e(F',L) = e(F,L)$. 

\begin{theorem} \label{thm:chrom-dual}
Let $F \subset \Si \times I$ be a connected spanning surface such that the map $H_1(F;\ZZ/2) \to H_1(\Si \times I;\ZZ/2)$ is surjective. Let $F'=F\#_\tau \Si$ be as above, and set  $\cK_F= \Ker (H_1(F;\ZZ/2) \to H_1(\Si \times I;\ZZ/2))$. Then the Brown invariant of $\varphi_{F'}$ is equal to the Brown invariant of the restriction of $\varphi_F$ to $\cK_F,$ i.e., $\Br(\varphi_{F'}) = \Br(\varphi|_{\cK_F})$. 
 \end{theorem}
 
\begin{proof}
Let $g =g(\Si)$ be the  genus of $\Si$. As in the proof of \cite[Theorem 4.1]{Boden-Chrisman-Karimi-2021}, we can choose a basis for $H_1(F';\ZZ)$ such that the Gordon-Litherland matrix has block decomposition 
\begin{equation} \label{eqn:big-matrix}
M=\begin{bmatrix} A&*& 0 \\ *&B & J_g \\ 0&J_g^\tr & 0 \end{bmatrix},
\end{equation}
where $A$ is the $n \times n$ matrix for the restriction of $\cG_{F}$ to $\Ker (H_1(F;\ZZ) \to H_1(\Si \times I;\ZZ))$, and 
$$J_g=\begin{bmatrix} 0& I_g \\ -I_g & 0 \end{bmatrix}$$ 
is the standard $2g \times 2g$ symplectic matrix representing the intersection form on $\Si$. (Here $I_g$ denotes the $g \times g$ identity matrix.)

The matrix in \eqref{eqn:big-matrix} is unimodular congruent to one of the form:
\begin{equation*}\label{eqn:big-matrix-2}
M'=\begin{bmatrix} A&0& 0 \\ 0&B & J_g \\ 0&J_g^\tr & 0 \end{bmatrix}.
\end{equation*}
Therefore, by \Cref{unimodular}, it follows that $\Br(\varphi_{F'}) = \Br(\varphi_{M})=\Br(\varphi_{M'})$. 

A non-singular $\ZZ/4$ quadratic form $\phi \co V \to \ZZ/4$ is said to be \textit{metabolic} if there exists a half-dimensional subspace $H \subset V$ such that $\phi$ vanishes on $H$. (In \cite[\S 4]{Matsumoto}, this is defined as \textit{split}.) By \cite[Lemma 4.1]{Matsumoto}, if $\phi$ is metabolic, then
$\Br(\phi)=0$. Since the Brown invariant is additive under orthogonal sum,
and since the quadratic form associated to 
$$\begin{bmatrix} B & J_g \\ J_g^\tr & 0 \end{bmatrix}$$ is clearly metabolic,
it follows that $\Br(\varphi_{M'})=\Br(\varphi_A).$ However, $\varphi_A$ is equal to the restricted $\ZZ/4$ quadratic form $\varphi|_{\cK_F},$ and the result now follows. 
\end{proof}
 
Now suppose $D \subset \Si$ is a cellularly embedded, checkerboard colorable link diagram. Then the inclusion map $i \co D \to \Si$ induces a surjection $i_* \co H_1(D;\ZZ) \to H_1(\Si;\ZZ).$ If $F$ is a checkerboard surface, then the map $H_1(F;\ZZ) \to H_1(\Si \times I;\ZZ)$ is surjective.
 
\begin{theorem} \label{thm:chromatic}
Let $D \subset \Si$ be a cellularly embedded, checkerboard colorable link diagram with coloring $\xi$,
and let $F$ be its black checkerboard surface. Let $\xi^*$ be the opposite coloring and $F^*$ its black checkerboard surface. (Thus $F^*$ is the white surface for $\xi.$)
The two sets of Brown invariants are related by chromatic duality:
$\Br_{F}(D) = \Br_{\xi^*}(D)$ and $\Br_{F^*}(D) = \Br_{\xi}(D)$.
 \end{theorem}

\begin{proof}
The proof follows from \Cref{thm:chrom-dual} and \cite[Lemma 5.3]{Boden-Chrisman-Karimi-2021}.
\end{proof}

\begin{remark}
At first glance, it may seem that the switch of colorings in \Cref{thm:chromatic} is the result of an incompatibility in the choice of conventions. However, it is unavoidable, and in fact it is an intrinsic aspect, and a curious one at that. This aspect is not readily apparent for classical links since the two sets of invariants are always equal in that case.
\end{remark}

Given a non-split checkerboard colorable link $L\subset \Si\times I$ with checkerboard surfaces $F$ and $F^*$, then any spanning surface for $L$ has Brown invariant equal to $\Br_F(L)$ or $\Br_{F^*}(L)$.

\begin{theorem}\label{cobordism}
Let $K \subset \Si\times I$ be a knot with spanning surface $F\subset \Si\times I$. Assume $e(F,K)=0$ and $\det(\cG_F) \not\equiv 0 \text{ (mod $2$)}$. If $K$ is slice, then $\Br_{F}(K)=0$.
\end{theorem}

\begin{proof}
If $K$ is slice, then \Cref{slice} implies that the Gordon-Litherland pairing is metabolic,
namely that $\cG_F$ vanishes on a half-dimensional subspace $U$ of $H_1(F;\ZZ)$. It follows that $\varphi_F$ must vanish on the image of $U$ under the mod 2 projection
$H_1(F;\ZZ) \to H_1(F;\ZZ/2)$. The image is
again a half-dimensional subspace of
$H_1(F;\ZZ/2)$.

Theorem 1.20 (ix) in \cite{Brown} implies that the Brown invariant vanishes on any non-singular, metabolic quadratic space (see also \cite[Lemma 4.1]{Matsumoto}). Therefore, $\Br(\varphi_F)=0$. Since $e(F,K)=0$, it follows that $\Br_{F}(K)=0$.
\end{proof}

\subsection{Mirror Images} 
In this section, we relate the
Brown invariants of a checkerboard colorable link in a thickened surface to those of its vertical and horizontal mirror images. We begin by considering orientation reversal on $L$.

If $L\subset \Si\times I$ is an oriented link with spanning surface $F \subset \Si \times I$, then $\Br_F(-L)=\Br_F(L),$ where $-L$ denotes the link with opposite orientation. Although $\Br_F(L)$ is insensitive to orientation change, it is sensitive to a change of orientation on a single component. In fact, writing $\Br_F(L) = \Br(\varphi_F) +\frac{1}{2} e(F,L)$, this is evident from the formulas $e(F,L)=e(F)-\la(L)$ and $\la(L) = \sum_{i\neq j} \lk(K_i,K_j)$ for
$L = K_1 \cup \cdots \cup K_n$.

Let $L\subset \Si\times I$ be an oriented link. The image of $L$ under the map $\phi \co \Si \times I \to \Si \times I$ given by $\phi(x,t) =(x,1-t)$ is called the \emph{vertical} mirror image of $L$ and is denoted $L^{*}$. 
Let $f\co \Si \to \Si$ be an orientation reversing homeomorphism and set
$\psi\co \Si \times I \to \Si \times I$ to be the map given by $\psi(x,t) =(f(x),t)$.
The image of $L$ under $\psi$ is called the \emph{horizontal} mirror image of $L$ and is denoted $L^\dag$. 

If $F \subset \Si \times I$ is a spanning surface for $L$, then
$F^*=\phi(F)$ is a spanning surface for $L^*$ and
$F^\dag=\psi(F)$ is a spanning surface for $L^\dag$.

\begin{proposition}\label{prop:mirror_image}
Let $L\subset \Si \times I$ be a  link with spanning surface $F \subset \Si \times I$.
Then the Brown invariant of the vertical and horizontal mirror images of $L$ satisfy
\begin{eqnarray*}
\Br_{F^*}(L^*)=-\Br_{F}(L), \; & \text{and} & \; \Br_{F^\dag}(L^\dag)=-\Br_{F}(L). 
\end{eqnarray*}

\end{proposition}

\begin{proof}
By \cite[Proposition 5.7]{Boden-Chrisman-Karimi-2021}, $\cG_{L^*,F^*} = - \cG_{L,F},$ and $e(F^*,L^*) = -e(F,L).$ The result now follows from \Cref{defn:Brown-link} and \cite[Theorem 1.20, (iii)]{Brown}. A similar argument shows that
$\cG_{L^\dag,F^\dag} = - \cG_{L,F},$ and $e(F^\dag,L^\dag) = -e(F,L),$ and the
 formula for $\Br_{F^\dag}(L^\dag)$ follows.  
\end{proof} 


\subsection{Arf invariants of links}
Let $L \subset \Si \times I$ be an oriented, null-homologous link and $F \subset \Si \times I$ a Seifert surface for $L$. Then the Gordon-Litherland pairing $\cG_F$ coincides with the symmetrized Seifert pairing. In particular, $\cG_F(a,a) \in 2 \ZZ$,
and its associated $\ZZ/4$-valued form $\varphi_F$ is even. Therefore, the Brown invariant of $\varphi_F$ is related to the Arf invariant of the ordinary quadratic form $q_F$ by the formula  $\Br(\varphi_F) = 4\Arf(q_{F})$.

Further, we have $e(F)=0$ and $e(F,L)=-\la(L) = -\sum_{i\neq j} \lk(K_i,K_j)$, the total linking number. (Here we write $L = K_1 \cup \dots \cup K_m$ as a union of its components.) Therefore, the link invariants are related by the formula
$$\Br_F(L)=4 \Arf(q_F) +\la(L).$$

In \cite{Micah-Sujoy}, Chrisman and Mukherjee define  Arf invariants for null-homologous knots in thickened surfaces. Indeed, for a knot $K\subset \Si \times I$  with Seifert surface $F$, they define $q_{K,F}(x)\equiv \lk(x^{+},x)\text{ (mod $2$)}$ and show that it gives an ordinary quadratic form on $H_1(F;\ZZ/2)$. Assuming that $q_{K,F}$ is non-singular, they 
use it to define the Arf invariant of $K$. Note that $q_{K,F}$ is non-singular precisely when $\det(K,F)\not \equiv 0 \text{ (mod $2$)}$.


We remark that, like the Brown invariants, the Arf invariant can be defined for singular, proper forms, i.e., the assumption that $q_{K,F}$ is non-singular is not necessary. This is illustrated in \Cref{ex-6-87310} below.
In general, the Brown invariant provides a way to define the Arf invariant for any oriented null-homologous link $L\subset \Si \times I$ and Seifert surface $F \subset \Si \times I$. 
Thus, the Brown invariants in \Cref{section-link} recover and extend the Arf invariants introduced in \cite{Micah-Sujoy}. 

\begin{proposition}
If $K\subset \Si \times I$ is a null-homologous knot with Seifert surface $F$,
and $\det(K,F)\not\equiv 0 \text{ (mod $2$)}$, then 
$$\Br_{F}(K)=4\Arf(q_{K,F}).$$ 
\end{proposition} 

\begin{proof}
By the definitions of $q_{K,F}$ and $\varphi_{F}$, we see that $\varphi_{F}=2q_{K,F}$. The result then follows from \cite[Theorem 1.20, (vii)]{Brown}. 
\end{proof}

\begin{example}\label{ex-6-87310}
\Cref{fig-6-87310} shows a six crossing knot $K$ in the thickened torus, along with the Tait graphs for the checkerboard surfaces $F$ (middle) and $F^*$ (right). 

For the surface $F$, the associated Goeritz matrix is 
$$G_{\xi}(D)=\left[
\begin{matrix}
1&-1&-1\\
*&1&-1\\
*&*&1
\end{matrix}
\right].$$ 
We have $\lambda(\varphi_{\xi})=4+4i$. Notice that $\varphi_\xi$ is singular and proper with $\Br(\varphi_{\xi})=1$. We have $\mu_{\xi}=-3$, therefore $\Br_{\xi}(K)=1-(-3)=4$.

For the surface $F^*$, the associated Goeritz matrix is 
$G_{\xi^*}(D)=[-3]$. Then $\lambda(\varphi_{\xi^*})=1+i$ and  $\Br(\varphi_{\xi^*})=1$. Further $\mu_{\xi^*}=-3$, thus $\Br_{\xi^*}(K)=1-(-3)=4$.

Notice that $K$ is null-homologous, in fact it is the knot $6.87310$ in Figure 21 of  \cite{Boden-Chrisman-Gaudreau-2020}. Let $F'$ be the Seifert surface whose Seifert matrices are listed in \cite[Table 3]{Boden-Chrisman-Gaudreau-2020}.
Then $F'$ is $S^*$-equivalent to $F$, and its quadratic form $q_{K,F'}$ is singular and proper. Since $\Br_{\xi}(K)=4$, it follows that $\Arf(q_{K,F'})=1.$
\hfill $\Diamond$ 
\end{example}

Suppose $K \subset \Si \times I$ is a  null-homologous knot with Seifert surface $F$. Notice that $q_{K,F}$ is non-singular if and only if $\det(K,F)$ is odd. Set $d = \det(K,F)$.
By Levine's formula (see \cite{Levine-66}), we have  
$$
\Arf(q_{K,F})=\begin{cases}0, & \text{if $d\equiv\pm 1$ (mod 8),}\\
1, & \text{if $d\equiv\pm 3$ (mod 8).}\end{cases}
$$

If, in addition, $K$ is slice, then \Cref{determinant} implies that $d$
is a perfect square. However, since $d$ is odd, it
follows that $$d=(2k+1)^2=4k^2+4k+1=4k(k+1)+1\equiv 1\text{ (mod $8$)}.$$
Therefore, Levine's formula implies that $\Arf(q_{K,F})=0$.
In particular, for null-homologous knots $K\subset \Si \times I,$ the knot determinant provides a stronger slice obstruction.

\begin{figure}[htbp]
\begin{center}
\includegraphics[scale=0.90]{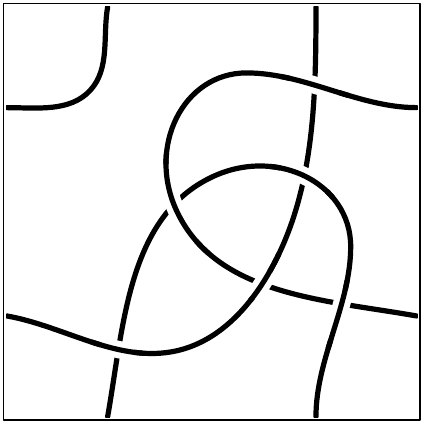}\hspace{.3cm}
\includegraphics[scale=0.90]{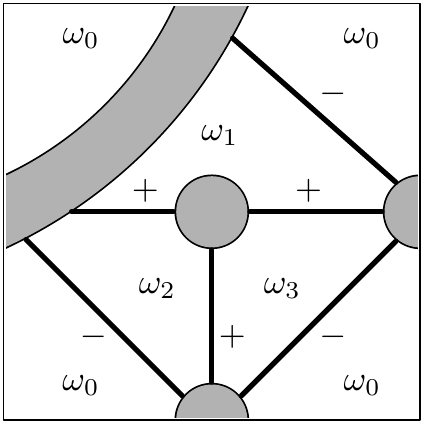}\hspace{.3cm}
\includegraphics[scale=0.90]{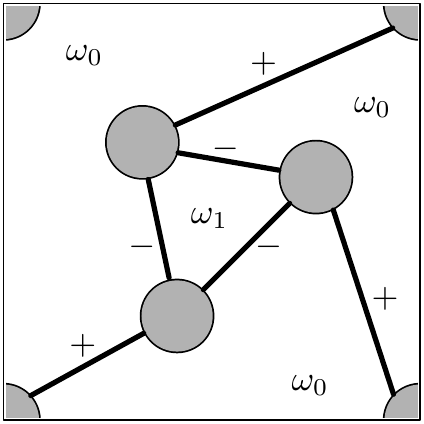}
\caption{\small A null-homologous knot in the thickened torus and the two Tait graphs for its checkerboard surfaces.}
\label{fig-6-87310}
\end{center}
\end{figure}


\section{Concordance of spanning surfaces} \label{sec5}  

In this section, we introduce a notion of concordance for spanning surfaces of knots in thickened surfaces. This is based on \cite[Definition 4.1.1]{Micah-Sujoy}, which is the corresponding notion of concordance for Seifert surfaces.
We show that the knot signature and Brown invariant are invariant under concordance of spanning surfaces.

\begin{definition} \label{defn-spanning-surf}
Let $K_0$ and $K_1$ be knots  in $\Si_0\times I$ and $\Si_1\times I$, respectively, with spanning surfaces $F_0 \subset \Si_0 \times I$ and $F_1 \subset \Si_1 \times I$.
So $F_0, F_1,$ are unoriented surfaces with boundary $\partial F_0 =K_0$ and $\partial F_1 =K_1$.

The spanning surfaces $F_0$ and $F_1$ are said to be \textit{concordant} if there exist:
\begin{itemize}
\item[(i)] a compact oriented 3-manifold $W$ with boundary $\partial W = -\Si_0 \cup \Si_1$, 
\item[(ii)] a properly embedded annulus $A \subset W \times I$ with boundary $\partial A = K_0 \cup K_1$, and \item[(iii)] a compact unoriented 3-manifold $V \subset W \times I$ with boundary $\partial V = F_0 \cup A \cup F_1$. 
\end{itemize}
\end{definition}

If $K_0 \subset \Si_0\times I$ and $K_1 \subset \Si_1\times I$ are knots with concordant spanning surfaces $F_0 \subset \Si_0\times I$ and $F_1 \subset \Si_1\times I$,
then the knots $K_0$ and $K_1$ are concordant. However, the converse is not true. In general, for a checkerboard colorable knot in a thickened surface $\Si \times I$ with genus $g(\Si) \geq 1$, the spanning surfaces given by the black and white regions will not be concordant in the above sense.

The proof of the following lemma is similar to that of \Cref{kernel}, and we omit it.

\begin{lemma}\label{kernel-2}
There are generating sets $\{\ga_1, \ldots, \ga_{g_0}\}$ for $H_1(F_0; \ZZ)$ and $\{\ga'_{1}, \ldots, \ga'_{g_1}\}$ for $H_1(F_1; \ZZ)$,
such that $\ga_1, \ldots, \ga_{m_0},\ga'_1,\ldots,\ga'_{m_1}$ lie in the kernel of the map $H_1(F_0;\ZZ)\oplus H_1(F_1;\ZZ) \to H_1(V;\QQ),$ where $m_0+m_1=\frac{1}{2}(g_0+g_1)$.
\end{lemma}

We use \Cref{kernel-2} to prove the following result, which is an analogue of \Cref{slice} for two knots with concordant spanning surfaces.

\begin{theorem}\label{concordance}
Let $K_0 \subset \Si_0\times I$ and $K_1 \subset \Si_1\times I$ be knots with spanning surfaces $F_0 \subset \Si_0 \times I$ and $F_1 \subset \Si_1 \times I$, respectively. If $F_0$ and $F_1$ are concordant and $n(K_0,F_0)=n(K_1,F_1)=0$, then
$$\si(K_0,F_0)=\si(K_1,F_1).$$
\end{theorem}

\begin{proof}
Define $\Theta$ on $H_1(F_0;\ZZ)\oplus H_1(F_1;\ZZ)$ as follows
$$\Theta((x_0,x_1),(y_0,y_1))=-\cG_{F_0}(x_0,y_0)+\cG_{F_1}(x_1,y_1).$$
By \Cref{kernel-2}, and similar to the proof of \Cref{slice}, $\ga_1, \ldots, \ga_{m_0},\ga'_1,\ldots,\ga'_{m_1}$
generate a totally isotropic subspace for $\Theta$. Since the nullity of $\Theta$ is zero, it follows that
$$2(m_0+m_1)=g_0+g_1\leq g_0+g_1-|\sig \Theta|,$$
so $0=\sig\Theta=-\sig\cG_{F_0}+\sig\cG_{F_1}$, which implies $\sig\cG_{F_0}=\sig\cG_{F_1}$.

Again, the self-intersection number of $E$ vanishes. The self-intersection number of $E$ equals
$$-\lk(K_0,K_{0}')+\lk(K_1,K_{1}')=-e(F_0,K_0)+e(F_1,K_1)=0.$$
It follows that $e(F_0,K_0)=e(F_1,K_1)$, and $\si(K_0,F_0)=\si(K_1,F_1).$
\end{proof}

\begin{theorem}\label{brown-concordance}
Let $K_0 \subset \Si_0\times I$ and $K_1 \subset \Si_1\times I$ be knots with spanning surfaces $F_0 \subset \Si_0 \times I$ and $F_1 \subset \Si_1 \times I$, respectively. If $F_0$ and $F_1$ are concordant and $\det(\cG_{F_0}),\det(\cG_{F_1}) \not \equiv 0 \text{ (mod $2$)}$, then
$$\Br_{F_0}(K_0)=\Br_{F_1}(K_1).$$
\end{theorem}

\begin{proof}
Consider the mod 2 reduction of the form $\Theta$ defined in the proof of \Cref{concordance}. Now define $\Phi:H_1(F_0;\ZZ/2)\oplus H_1(F_1;\ZZ/2) \to \ZZ/4$, as $\Phi([a],[b])\equiv -\cG_{F_0}(a,a)+\cG_{F_1}(b,b)\text{ (mod $4$)}$. We can check that $\Phi$ is a quadratic enhancement, and we can define $\Br(\Phi)$.
  
Similar to \Cref{cobordism}, we deduce that $\Br(\Phi)=0$. Notice that the Brown invariant is additive, and that $\Br(-\varphi)\equiv -\Br(\varphi) \text{ (mod $8$)}$. The result now follows. 
\end{proof}

The next result shows if two knots in thickened surfaces are concordant, then under some mild hypotheses, their spanning surfaces are concordant. As a result, Theorems \ref{concordance} and \ref{brown-concordance} hold more generally for some concordant knots.

\begin{theorem}
Let $K_0\subset \Si_0\times I$ and $K_1\subset \Si_1\times I$ be knots with spanning surfaces $F_0 \subset \Si_0 \times I$ and $F_1 \subset \Si_1 \times I$, respectively.
If $K_0$ and $K_1$ are concordant, then there exists an oriented 3-manifold $W$ with $\partial W =\Si_1 \sqcup -\Si_0$ and a properly embedded annulus $A$ in $W \times I$ with $\partial A = K_0 \sqcup K_1$. Let $E=F_0\cup A\cup F_1$. If 
$[E]=0$ in $H_2(W\times I; \ZZ/2)$, then $F_0$ and $F_1$ are concordant.
\end{theorem}

\begin{proof}
We can arrange that $e(F_0,K_0)=e(F_1,K_1)$, so $e(E)=0$. Since $[E]=0,$ then \Cref{boundary} implies that there exists a compact 3-manifold $V\subset W\times I$ with $\partial V=E$.
The result now follows.  
\end{proof}

In closing we mention a few open problems and questions for future research. 

It would be interesting to develop a theory of algebraic concordance for checkerboard colorable knots in thickened surfaces. Here, we expect that the notion of concordance of spanning surfaces could be useful, cf.~\cite{Micah-Sujoy}. What kind of torsion does this new algebraic concordance group contain?

It would also be interesting to prove more general results about concordance invariants of the signature and Brown invariants for links in thickened surfaces.
A promising approach would be to develop an interpretation of the Brown invariants in terms of Pin-structures and show they are invariants of Pin bordism, cf. \cite{Kirby-Taylor}.

\subsection*{Acknowledgements}
The first author was partially funded by the Natural Sciences and Engineering Research Council of Canada.
The authors are grateful to Andrew Nicas, Danny Ruberman, and Will Rushworth for their input and feedback, and to Lindsay White for her computational help.
We would also like to thank Colin Bijaoui and Marco Handa for their initial work on this project; and Micah Chrisman for sending comments on an earlier draft.  


\end{document}